\documentclass[11pt,twoside,a4paper]{amsart}
\usepackage{amssymb}
\date{\today}

\date{\today}
\usepackage{geometry}
\def\m{{\mathfrak m}}

\def\1{{\bf 1}}

\def\ann{{\rm ann}}
\def\deg{\text{deg}\,}

\def\w{\wedge}

\def\dbar{\bar\partial}

\def\C{{\mathbb C}}
\def\w{{\wedge}}
\def\P{{\mathbb P}}

\def\supp{\text{supp}}

\def\M{{\mathcal M}}

\def\S{{\mathcal S}}

\def\W{{\mathcal W}}

\def\I{{\mathcal I}}

\def\reg{{\rm reg\,}}

\def\Hom{{\rm Hom }}
\def\codim{{\rm codim\,}}

\def\Im{{\rm Im\, }}

\def\K{{\mathcal K}}
\def\Ker{{\rm Ker\,  }}

\def\Ok{{\mathcal O}}

\def\L{{\mathcal L}}

\def\L{{\mathcal L}}
\def\U{{\mathcal U}}
\def\ann{{\rm ann\,}}

\def\J{{\mathcal J}}
\def\nbh{neighborhood }

\def\be{\begin{equation}}
\def\ee{\end{equation}}
\def\PM{{\mathcal{PM}}}
\def\Pk{{\mathbb P}}
\def\CH{{\mathcal{CH}}}

\def\pmm{{pseudomeromorphic }}

\newtheorem{thm}{Theorem}[section]
\newtheorem{lma}[thm]{Lemma}

\newtheorem{prop}[thm]{Proposition}

\theoremstyle{definition}

\theoremstyle{remark}

\newtheorem{preremark}[thm]{Remark}
\newtheorem{preex}[thm]{Example}

\newenvironment{remark}{\begin{preremark}}{\qed\end{preremark}}
\newenvironment{ex}{\begin{preex}}{\qed\end{preex}}

\numberwithin{equation}{section}

\title{A global Brian\c con-Skoda-Huneke-Sznajdman theorem}

\begin{document}

\date{\today}



\author{Mats Andersson}

\address{Department of Mathematics\\Chalmers University of Technology and the University of
Gothenburg\\S-412 96 G\"OTEBORG\\SWEDEN}

\email{matsa@chalmers.se}


\thanks{The author was
  partially supported by the Swedish
  Research Council.}

\begin{abstract} 
We prove a global effective membership result for polynomials
on a non-reduced algebraic subvariety of $\C^N$. It can be seen
as a global version of a recent local result of
Sznajdman, generalizing the Brian\c con-Skoda-Huneke theorem
for the local ring of holomorphic functions at a point on a reduced analytic
space. 
\end{abstract}

\maketitle

\section{Introduction}\label{puma}
Let $x$ be a point on a smooth analytic variety $X$ of pure dimension $n$ and let $\Ok_x$ be the local ring of 
holomorphic functions. The classical Brian\c con-Skoda theorem, \cite{BS},  states that if 
 $(a)=(a_1,\dots, a_m)$ is any ideal in $\Ok_x$ and $\phi$ is in $\Ok_x$, then $\phi\in(a)^r$
if 
\begin{equation}\label{bshvillkor}
|\phi|\le C|a|^{\nu+r-1}
\end{equation}
holds with $\nu=\min(m,n)$. The proof  given 
in \cite{BS} is purely analytic.  However, the condition \eqref{bshvillkor} is equivalent to 
saying that
$\phi$ belongs to the the integral closure
$\overline{(a)^{\nu+r-1}}$, and thus the theorem admits a purely algebraic formulation. 
Therefore  it was somewhat astonishing that it took several years before 
algebraic proofs were found, \cite{LS,LT}. Later on,  Huneke, \cite{Hun}, proved  
a far-reaching algebraic
generalization which contains the following statement for non-smooth $X$.

\smallskip\noindent
{\it Let $x\in X$ be a point on a reduced analytic variety of pure dimension.
There is a number $\nu$ such 
that if $(a)=(a_1,\dots, a_m)$ is any ideal in $\Ok_x$ and $\phi$ is in $\Ok_x$, then
\eqref{bshvillkor}
implies that $\phi\in (a)^r$.} 
\smallskip

An important point  is that $\nu$ is uniform  with respect to both $(a)$ and  $r$.
The smallest possible such $\nu$ is called the  Brian\c con-Skoda number, and it depends 
on the complexity of the singularities of $X$ at $x$.  
An analytic proof of this statement  appeared in \cite{ASS}.
A nice variant for a non-reduced $X$ of pure dimension
is formulated and proved in \cite{Sz}. 
 
Let $x$ be a point on a non-reduced analytic space $X$ of pure dimension $n$,  and let
$X_{red}$ be the underlying reduced space,
cf., Section \ref{svart0} below.  There is a natural surjective mapping
$\Ok_{X,x}\to\Ok_{X_{red},x}$. Let
$i\colon X\to\Omega\subset\C^N$ be a local embedding, and let $\J_{X,x}$ be the associated
local ideal in $\Ok_{\Omega,x}$, so that $\Ok_x=\Ok_{X,x}=\Ok_{\Omega,x}/\J_{X,x}$. 
A  holomorphic differential operator $L$ in $\Omega$ is Noetherian  
at $x$ if $L\phi$ vanishes
on $X_{red,x}$ (or equivalently, $L\phi\in\sqrt{\J_{X,x}}=\J_{X_{red},x}$) for all
$\phi\in\J_{X,x}$.
Such an $L$ defines an intrinsic  mapping
$$
L\colon \Ok_{X,x}\to \Ok_{X_{red},x}, \quad \phi\mapsto L\phi.
$$
\begin{thm}[Sznajdman, \cite{Sz}]\label{szthm}
Given $x\in X$, there is a finite set $L_\alpha$ of Noetherian operators 
at $x$ and a number $\nu$ such that for each ideal $(a)=(a_1,\ldots,a_m)\subset\Ok_{X,x}$ and
$\phi\in\Ok_{X,x}$,
\begin{equation}\label{bergen}
|L_\alpha\phi|\le C|a|^{\nu+r}\  {\rm on}\ X_{red,x} 
\end{equation}
for all $\alpha$, implies that  $\phi\in(a)^r$.
\end{thm}

Here $|a|$  means  $|a_1|+\cdots + |a_m|$  (where $|a_j|$ is the modulus of the image
of $a_j$ in $\Ok_{X,x}$), which up to  constants  is independent of the choice
of generators of the ideal  $(a)$.  The condition  \eqref{bergen} means that 
$L_\alpha\phi$ is in the integral closure of the image 
in $\Ok_{X_{red},x}$ of $(a)^{\nu+r}$.  

Applying to $(a)=(0)$ we find that $L_\alpha\phi=0$ on $X_{red,x}$ for all $\alpha$ implies
that $\phi=0$ in $\Ok_{X.x}$. 

We now turn our attention to global variants.
Let $V$ be a purely  $n$-dimensional algebraic subvariety of $\C^N$ and let 
$J_V\subset\C[x_1,\ldots,x_N]$ be the associated ideal. 
Assume that
$F_j$ are polynomials in  $\C^N$ of degree $\le d$.
If the polynomial
$\Phi$ belongs to the restriction of the ideal $(F_1,\ldots,F_m)$ to $V$, i.e., there are polynomials $Q_j$ such that
\begin{equation}\label{hummer}
\Phi=\sum_1^m F_jQ_j  + J_V ,
\end{equation}
then it is natural to ask for a representation \eqref{hummer}
with some control of the degree of $Q_j$.
It is well-known that if $V=\C^N$, then in general 
$\max_j\deg F_jQ_j$ must be doubly exponential in $d$, i.e., like $2^{2^d}$.  However,
in the Nullstellensatz, i.e., $\Phi=1$,  then (roughly speaking) $d^n$ is enough, 
this is due to Koll\'ar, \cite{Koll}, and Jelonek, \cite{Jel}.
In \cite{Hick} Hickel proved a global effective version of the Brian\c con-Skoda theorem
for polynomial ideals in $\C^n$, 
basically saying that if $|\Phi|/|F|^{\min(m,n)}$ is locally bounded, 
then there is a representation \eqref{hummer} in $\C^n$
with $\deg F_jQ_j\le \deg\Phi + Cd^n$. For the precise statement, see \cite{Hick} or \cite{AWsemester}.
In  \cite[Theorem A]{AWsemester} a generalization to polynomials on reduced algebraic subvarieties
of $\C^N$ appeared.   Our objective in this paper is to
find a generalization to a not necessarily reduced algebraic subvariety $V$ of $\C^N$
of pure dimension $n$.  

Let $X$ be the closure (see Section~\ref{svart00})  of $V$ in $\P^N$ and let $X_{red}$ be the underlying reduced variety.
Given polynomials  $F_1,\ldots, F_m$, let $f_j$ denote the corresponding $d$-homogenizations,
considered as sections of the line bundler $\Ok(d)|_{X_{red}}$,  and let 
$\J_f$ be the coherent analytic sheaf on $X_{red}$ generated by $f_j$.
Furthermore, 
let $c_\infty$ be the maximal codimension of  the
so-called  {\it distinguished varieties} of the sheaf $\J_f$,
in the sense of Fulton-MacPherson,   
that are contained in 
$$
X_{red,\infty}:=X_{red}\setminus V_{red},
$$
see Section~\ref{grotesk}.
It is well-known that the codimension of a distinguished variety
cannot exceed the number $m$, see, e.g., \cite[Proposition~2.6]{EL}, and thus
$$
c_\infty\le\min(m,n).
$$
We let $Z_f$ denote the zero variety of $\J_f$ in $X_{red}$.

Let  $\reg X$ denote the  so-called {\it (Castelnuovo-Mumford) regularity} 
of $X\subset\P^N$, see Section~\ref{svart00} below.  
We can now formulate the main result of this paper.

\begin{thm}[Main Theorem]\label{huvudsats}
 Assume that  $V$ is an algebraic subvariety  of $\C^N$ of pure
dimension $n$ and let $X$ be its closure in $\P^N$.  There is a finite set of holomorphic
differential operators $L_\alpha$ on $\C^N$ with polynomial coefficients and 
a number $\nu$ so that the following holds:

\smallskip
\noindent  (i) For each point $x\in V$ the germs of $L_\alpha$ are Noetherian
operators at $x$ such that the conclusion in Theorem~\ref{szthm} holds.

\smallskip
\noindent  (ii)  If $F_1,\ldots,F_m$ are polynomials of degree
$\le d$, $\Phi$ is a polynomial,  and
\begin{equation}\label{kaka}
|L_\alpha\Phi| /|F|^{\nu} \text{ is locally bounded on } V_{red} 
\end{equation}
for each $\alpha$, 
then there are polynomials $Q_1,\ldots, Q_m$ such that 
 \eqref{hummer} holds 
and 
\begin{equation}\label{chimpans}
\deg (F_j Q_j)\leq 
\max \big(\deg \Phi + \nu d^{c_\infty} \deg X_{red}, (d-1)\min(m,n+1) +\reg X\big).
\end{equation}
\end{thm}

If there are no 
distinguished varieties of $\J_f$ contained in $X_{red,\infty}$,  then 
$d^{c_\infty}$ shall be interpreted as $0$.
\smallskip

In case $V$ is reduced we can choose $L_\alpha$ as just the identity; then
(ii) is precisely (part (i) of) Theorem~A in \cite{AWsemester}. If $V=\C^n$
we get back Hickel's theorem, \cite{Hick},   mentioned above.

\begin{ex}  
If we apply  Theorem~\ref{huvudsats} to  Nullstellensatz  data,
i.e., $F_j$ with no common zeros on $V$
and $\Phi=1$, then the hypothesis \eqref{kaka} is fulfilled, and we
thus get $Q_j$ such that $F_1Q_1+\cdots +F_mQ_m -1$ belongs to $J_V$  and
$$
\deg(F_jQ_j)\le
\max\big(\nu d^{c_\infty} \deg X_{red}, (d-1)\min(m,n+1) +\reg X\big).
$$
See \cite[Section~1]{AWsemester} for a discussion of this estimate in the reduced case.
\end{ex}

\begin{ex}\label{macex}  
 If $f_j$ have no common zeros on $X$ and $\Phi$ is any polynomial,  %
then there is a solution to \eqref{hummer}
such that
$$
\deg F_jQ_j\le \max(\deg \Phi, (d-1)(n+1)+\reg X).
$$
If  $X=\P^n$, then $\reg X=1$ and so we get back the
classical Macaulay theorem.
\end{ex}

\begin{remark}
It follows that $L_\alpha$ is a set of Noetherian operators such that a polynomial
$\Phi\in \C[x_1,\ldots,x_N]$ is in $J_V\subset\C[x_1,\ldots,x_N]$
if and only $L_\alpha\Phi=0$ on $V_{red}$ for each $\alpha$.
The existence of such a set is well-known, and 
a key point in the celebrated Ehrenpreis-Palamodov fundamental theorem, 
\cite{Ehr} and \cite{Pal}; see also, e.g., \cite{Bj0} and \cite{Ob}.
\end{remark}

\begin{remark} It turns out, see Theorem~\ref{skohorn} below, that the Noetherian operators
$L_\alpha$ in Theorem~\ref{huvudsats} have the following additional property:
For each $\alpha$ there is a finite set of holomorphic differential operators $M_{\alpha,\gamma}$
such that
\begin{equation}\label{ormvrak}
L_{\alpha}(\Phi\Psi)=\sum_{\gamma} L_\gamma\Phi\M_{\alpha,\gamma}\Psi
\end{equation}
for  any holomorphic functions  $\Phi$ and $\Psi$. This formula shows that set 
of functions that satisfy \eqref{bergen} at a point $x$ is indeed an ideal.
\end{remark}

By homogenization, this kind of effective results can be reformulated 
as geometric statements:  Let
 $z=(z_0,\ldots,z_N)$,  $z'=(z_1,\ldots,z_N)$, let
 $f_{i}(z):=z_0^{d}F_i(z'/z_0)$ be the $d$-homogenizations of $F_i$, considered
as sections of $\Ok(d)\to\P^N$, and let
$\varphi(z):=z_0^{\deg \Phi}\Phi(z'/z_0)$.
Then there is a representation \eqref{hummer} on $V$ with $\deg(F_jQ_j)\le\rho$
if and only if there are  sections $q_i$ of $\Ok(\rho-d)$ on $\P^N$ such that
\begin{equation}\label{hoppla}
f_{1} q_1+\cdots +f_mq_m=z_0^{\rho-\deg\Phi}\varphi
\end{equation}
on  $X$ in  $\P^N$; that is, the difference of the right and the left hand sides
belongs to the sheaf $\J_X$.

\smallskip

To prove Theorem~\ref{huvudsats} we first have to define a suitable set of global Noetherian operators
on $\Pk^N$.
This is done in Section~\ref{pyts}
following the ideas of Bj\"ork, \cite{Bj},  in the local case, starting
from  a representation of $\J_X$ as the annihilator of a tuple
of so-called Coleff-Herrera currents on $\Pk^N$. 
The rest of the proof of Theorem~\ref{huvudsats}, given in Section~\ref{grotesk},  
follows to a large extent the proof of Theorem~A in \cite{AWsemester}.
By the construction in  \cite{AW1} we have a residue current  $R^X$ associated with $\J_X$ such that
the annihilator ideal of  $R^X$ is precisely  $\J_X$.  Following the ideas in \cite{AWsemester} 
we then form the ``product'' $R^f\w R^X$, where $R^f$ is the current of Bochner-Martinelli type
introduced in \cite{A2}, inspired by \cite{PTY}. By computations as in \cite{Sz},
the condition \eqref{kaka} ensures that
$\phi$ annihilates this current at each point $x\in V_{red}$. 
If $\rho$ is large enough, 
this is reflected by the first entry of the right hand side of \eqref{chimpans}, then
a  geometric estimate from \cite{EL} ensures that the $\rho$-homogenization
$\phi$ of $\Phi$ indeed satisfies a condition
like \eqref{kaka} even at infinity. Therefore $\phi$ annihilates the current
$R^f\w R^X$ everywhere on $\Pk^N$.  For this argument it is important that the
Noetherian operators extend to $\Pk^N$.  
The proof of Theorem~\ref{huvudsats}  is then concluded along the same lines as in 
\cite{AWsemester} by solving a sequence of $\dbar$-equations. If $\rho$
is large enough, this is reflected by the second entry in the right hand side of
 \eqref{chimpans}, there are no cohomological obstructions.
We then get a global representation of 
$\phi$ as a member of $\Ok(\rho)\otimes (\J_f+\J_X)$.
After dehomogenization we get the desired representation \eqref{hummer}.

In Section~\ref{bakgrund} we collect some necessary background material.
In Section  \ref{svart2} we discuss global Coleff-Herrera currents on
projective space. As mentioned above, the proof of our main theorem is 
given in the last two sections.

\smallskip 
\textbf{Acknowledgement}:  We would like to thank the referee for careful
reading and several suggestions to improve the presentation.

\section{Preliminaries}\label{bakgrund}

In this section we collect various definitions and facts that will be used later on.

\subsection{Non-reduced analytic space}\label{svart0}
A reduced analytic space $Z$  
is locally described as an analytic subset of some open set $\Omega\subset\C^N$,
and the sheaf $\Ok_Z$ of holomorphic functions on $Z$, the {\it structure sheaf},
is then isomorphic to $\Ok_\Omega/\J_Z$, where $\J_Z$ is the ideal sheaf of functions in $\Omega$ that vanish on $Z$.  
A non-reduced analytic space $X$ (also referred to as an analytic scheme)
with underlying
reduced space  $Z$ and {\it structure sheaf}  $\Ok_X$ is locally of the form $\Ok_X=\Ok_\Omega/\J$, where $\J\subset\J_Z$ is a
coherent ideal sheaf with common zero set $Z$.  
Thus $\J_Z=\sqrt{\J}$ and $\Ok_Z$ is obtained from $\Ok_X$ by taking the quotient by all nilpotent elements in $\Ok_X$.  Given the non-reduced space $X$ we 
denote the underlying reduced space by $X_{red}$.   

The space $X$ has {\it pure dimension}
$n$ if for each $x\in X_{red}$,
all the associated prime ideals of the local ring $\Ok_x$ has dimension $n$.
In particular, then $X_{red}$ has pure dimension $n$.  

\subsection{Algebraic and projective spaces}\label{svart00}
We will only be concerned with analytic spaces that are globally embedded in
some $\C^N$ or $\P^N$. 
An analytic subspace $V\subset\C^N$ is {\it algebraic}
if the sheaf $\J_V$ is generated by a finite number of polynomials.  Let $J_V$ be the 
corresponding ideal in the polynomial ring $\C[x_1,\ldots, x_N]$.  Let $J_X$ be the
homogeneous ideal in the graded ring $\C[x_0,\ldots,x_N]$ generated by homogenizations
of the elements in $J_V$.  If $J_V$ has pure dimension $n$, then $J_X$ has pure dimension $n+1$.  In particular, $0$ is not an associated prime ideal. 
Each homogeneous polynomial corresponds to a global section of the line bundle
$\Ok(\ell)\to \Pk^N$ for some $\ell$. These sections define a coherent analytic sheaf $\J_X$ 
over $\Pk^N$ of pure dimension $n$. We define the closure $X$ of $V$ as the analytic
subspace of $\Pk^N$ with structure sheaf $\Ok_X=\Ok_{\Pk^N}/\J_X$. 
 It is clear that the sheaf $\J_X$ coincides with  the sheaf $\J_V$ defined by the
ideal $J_V$  in $\C^N$.


Let $S$ be the graded ring $\C[x_0,\ldots,x_N]$ and let $S(-d)$ be the
$S$-module that is equal to $S$ but with the gradings shifted by $d$. 
Let $J_X$ be the homogeneous ideal in $S$ of all forms that belong to
$\J_X$. Since $0$ is not an associated prime ideal of $J_X$, cf., \cite[Corollary 20.14]{Eis}, 
see also \cite[Section 2.7]{AWsemester},
 there is a graded free resolution 
\begin{equation}\label{plast}
   0\to \oplus_1^{r_N}S(-d_N^{i})\stackrel{c_N}{\to}      \ldots         \stackrel{c_2}{\to}      \oplus_1^{r_1}S(-d_1^{i})  \stackrel{c_1}{\to} S\to S/J_X\to 0
\end{equation}
of the $S$-module $S/J_X$, 
where $c_k=(c_k^{ij})$ are matrices of homogeneous forms in $\C^{N+1}$
with $\deg c_k^{ij}= d^j_k-d^{i}_{k-1}$. 
The number
\begin{equation}\label{cn}
\reg X:= \max_{k,i} (d_k^{i}-k)+1
\end{equation}
is called the Castenouvo-Mumford regularity of $X$ in $\P^N$, see, e.g.,
\cite{Eis2}. This number
describes the complexity of the embedding of $X$ in $\P^N$;
thus two isomorphic analytic spaces embedded in different ways may
have different regularities.

\subsection{Some residue theory}\label{svart}
Let $Y$ be a (smooth) complex manifold of dimension $N$.
Given a holomorphic function $f$ on $Y$, following Herrera and Lieberman, \cite{HeLi}, 
one can define the principal value current
$1/f$ as the limit
$$
\lim_{\epsilon\to 0}\chi(|f|^2v/\epsilon)\frac{1}{f},
$$
where $\chi(t)$ is the characteristic function of the
interval $[1,\infty)$ or a smooth approximand and $v$ is any smooth
strictly positive function.
The existence of this limit for a general $f$ relies on
Hironaka's theorem that ensures that there is a modification $\pi\colon\widetilde Y\to Y$
such that $\pi^*f$ is locally a monomial. 
It is readily checked that 
$f(1/f)=1$ and $f\dbar(1/f)=0$.
 The current $1/f$ is well-defined even if $f$ is a holomorphic section
of a Hermitian line bundle over $Y$, since $a(1/af)=1/f$ if $a$ is holomorphic and
nonvanishing.   

\begin{ex}\label{brus} 
In one complex variable it is quite elementary to see that the principal value
current $1/s^{m+1}$ exists and that 
$$
\dbar\frac{1}{s^{m+1}}\w ds.\xi=\frac{2\pi i}{m!}\frac{\partial^m}{\partial s^m}\xi(0),
$$
for test functions $\xi$.
\end{ex}

The sheaf $\PM=\PM_Y$ of {\it pseudomeromorphic currents}, introduced in \cite{AW2,AS},
consists of currents on $Y$ that are finite sums of direct images under (compositions of)
modifications, simple projections and open inclusions of 
currents of the form 
$$
\frac{\xi}{s_1^{\alpha_1}\cdots s_{\ell-1}^{\alpha_{\ell-1}}}\w
\dbar\frac{1}{s_\ell^{\alpha_\ell}}\w\ldots\w \dbar\frac{1}{s_m^{\alpha_m}},
\quad m\le n,
$$
in some $\C^m_s$ and $\xi$ is a smooth form with compact support.

The sheaf $\PM$ is closed under $\dbar$ (and $\partial$) and multiplication by smooth forms.
If $\tau$ is in $\PM$ and has support on an analytic subset $V\subset Y$ and $\eta$ is
a holomorphic form that vanishes on $V$, then 
\begin{equation}\label{mellom}
\overline{\eta}\w\tau=0,\quad 
d\bar\eta\w\tau=0. 
\end{equation}
The first equality roughly speaking means that  $\tau$ does not involve 
anti-holomorphic derivatives.
By a standard argument the second equality in \eqref{mellom}  implies:

\smallskip
\noindent {\it Dimension principle:  If $\tau$ is a pseudomeromorphic current
on $Y$ of bidegree $(*,p)$ that has support on an analytic subset
$V$ of codimension $>p$, then $\tau=0$.}

\smallskip
Let $\U\subset Y$ be an open subset. If $\tau$ is in $\PM(\U)$ and $V\subset \U$
 is an analytic subvariety,  then the natural restriction
of $\tau$ to the open set $\U\setminus V$ has a canonical extension
as a principal value to a pseudomeromorphic current  ${\bf 1}_{\U\setminus V}\tau$ on $\U$. 
If  $h$ is  a holomorphic tuple in $\U$ with common zero set
$V$, and $\chi$ is a smooth approximand $\chi$ of the characteristic function
of  the interval $[1,\infty)$, then
\begin{equation}\label{smorgas}
{\bf 1}_{\U\setminus V}\tau=\lim_{\epsilon\to 0}\chi(|h|^2/\epsilon)\tau.
\end{equation}
It follows that  ${\bf 1}_{V}\tau:=\tau-{\bf 1}_{\U\setminus V}\tau$ is
pseudomeromorphic in $\U$ and has  support on $V$. 
Notice that if $\alpha$ is a smooth form, then
${\bf 1}_{V}\alpha\w\tau=\alpha\w{\bf 1}_{V}\tau.$
Moreover, 
if $\pi\colon \widetilde \U\to \U$ is a modification,  $\tilde\tau$ is in $\PM(\widetilde \U)$,
and
$\tau=\pi_*\tilde\tau$, then
$$
{\bf 1}_V\tau=\pi_*\big({\bf 1}_{\pi^{-1}V}\tilde\tau\big)
$$
for any analytic set $V\subset \U$.
For any analytic sets $W, W'\subset\U$,
$$
{\bf 1}_W{\bf 1}_{W'}\tau={\bf 1}_{W\cap W'}\tau.
$$


Let $Z\subset Y$ be an analytic subset of pure codimension $p$ and 
let $\tau$ be a \pmm current of bidegree $(N,*)$ with support on $Z$.
We say that $\tau$ has the {\it standard extension property}, SEP,
with respect to $Z$ if $\1_V\tau=0$ for each subvariety $V\subset Z\cap\U$ of positive
codimension, where $\U\subset Y$ is some open subset. 
The sheaf of such currents is denoted by 
$\W^Z$.  If $Z=Y$
we write $\W$ rather than  $\W^Y$.  
The subsheaf of $\W^Z$ of $\dbar$-closed currents of bidegree $(N,p)$
is called the sheaf of Coleff-Herrera currents\footnote{We adopt
here the convention from \cite{Bj}; in, e.g.,  \cite{Sz} these currents have
bidegree $(0,p)$.}, $\CH^Z$, on  $Z$.

\begin{remark}  The sheaf $\CH^Z$ was introduced by Bj\"ork, in a slightly different way.
For the equivalence to the definition given here,  see \cite[Section~5]{A9}.
\end{remark}

\begin{ex} Let $[Z]$ be the Lelong current associated with $Z$ and let $\beta$ be a
smooth form of bidegree $(p,*)$. Then $\mu=\beta\w[Z]$ is in $\W^Z$. If $\beta$ is holomorphic,
then $\mu$ is in $\CH^Z$. See, e.g., \cite[Example~4.2]{A9}.
\end{ex}

\begin{prop}\label{spadtag} 
If $\L$ is a holomorphic differential operator and  $\tau$ is in $\W^Z$,
then $\xi\mapsto\tau. \L\xi$ defines a current in $\W^Z$.
\end{prop}

\begin{proof} It is a local statement so by induction it is enough to 
let $\L$ be a partial derivative $\partial/\partial \zeta_1$ with respect to some
local coordinate system. Let $L$ denote the Lie derivative with respect to 
this vector field. Since $\xi$ has bidegree $(0,*)$, $(\partial/\partial \zeta_1)\xi=L\xi$.
Thus
$$
\tau.(\partial/\partial \zeta_1)\xi=\tau. L\xi=\pm L\tau.\xi,
$$
and $L\tau$ is in $\W^Z$ according to \cite[Theorem~3.7]{AW3}.  
\end{proof}

\subsection{Almost semi-meromorphic currents}\label{prutman}
We say that a current $b$ on a  smooth manifold $Y$ is {\it almost semi-meromorphic},
$b\in ASM(Y)$,  if there is a modification $\pi\colon Y'\to Y$, 
a holomorphic generically non-vanishing 
section $\sigma$ of a line bundle
$L\to Y'$ and an $L$-valued smooth form $\omega$ such that 
\begin{equation}\label{oskar}
b=\pi_*\frac{\omega}{\sigma},
\end{equation}
where $\omega/\sigma$ denotes the principal value current.  
This class of currents was introduced in \cite{AS} and studied in more detail in \cite{AW3}.
All results in this subsection can be found in the latter reference.

Let $ZSS(b)$, the Zariski singular support of $b$,
be the smallest analytic set such that $b$ is smooth in the complement. 

We will need the following results.

\begin{prop}[\cite{AW3}, Theorem~4.26]\label{jumper1}
If $b$ is almost semi-meromorphic on $Y$ and $\L$ is a holomorphic differential operator,
then $\L b$ is almost semi-meromorphic as well.
\end{prop}

Clearly,  $ZSS(\L b)\subset ZSS(b)$.  

\begin{thm}[\cite{AW3}, Theorem~4.8]\label{jumper2} If $b\in ASM(Y)$ and
$\tau$  is any \pmm current in $Y$, then there is a unique current
$T$ in $Y$ that coincides with $b\w \tau$ outside $ZSS(b)$ and such that
$\1_{ZSS(b)} T=0$.
\end{thm}

We will denote the extension $T$ by $b\w\tau$ as well. 
It follows from \eqref{smorgas} that
\begin{equation}\label{blus2}
b\w\tau=\lim_\delta \chi_\delta b\w\tau
\end{equation}
if  $\chi_\delta=\chi(|g|^2/\delta)$ where $g$ is a holomorphic tuple whose zero set
is precisely $ZSS(b)$.
It is  not hard to check, cf.,  \cite[Proposition 4.9]{AW3}, that if $V$ is any analytic set, then
\begin{equation}\label{jumper3}
\1_V (b\w \tau)=b\w\1_V\tau.
\end{equation}
It follows from \eqref{jumper3} that $b\in ASM(Y)$ induces a mapping
$$
\W^Z\to\W^Z, \quad \tau\mapsto b\w\tau.
$$
Given $a\in ASM(Y)$ and $\tau\in\PM^Y$ we define
$$
\dbar a\w\tau:= \dbar (a\w\tau)-(-1)^{\deg a} a\w\dbar\tau
$$
The definition is made so that the formal Leibniz rule holds. 

\begin{remark}
Clearly $\dbar a= b+ r(a)$ where
$b=\1_{X\setminus ZSS(a)}\dbar a$ and $r(a)$, the residue of $a$, has support on $ZSS(a)$.  One can check, cf., \cite[Proposition 416]{AW3},  that in fact $b\in ASM(X)$.  Thus we can define
$r(a)\w\tau:=\dbar a\w\tau-b\w a.$
If $\chi_\delta$ is as above,  then
\begin{equation}\label{hoppsan}
r(a)\w\tau=\lim_\delta \dbar\chi_\delta \w a\w \tau.
\end{equation}
\end{remark}

If $a$ is holomorphic outside $ZSS(a)$, then clearly
the support of $\dbar a\w\tau$ is contained in $\supp \tau\cap ZSS(a)$.
In particular, if $\gamma_1,\ldots,\gamma_p$ are holomorphic functions, 
 then by induction we can form the current
\begin{equation}\label{chprodukt}
\dbar\frac{1}{\gamma_p}\w\cdots\w\dbar\frac{1}{\gamma_1}.
\end{equation}
Clearly it is $\dbar$-closed and has support on $Z_\gamma=\{\gamma_1=\cdots=\gamma_p=0\}$. 
If in addition   
$Z_\gamma$ has codimension $p$, then \eqref{chprodukt}  is
anti-commuting in its factors, see, e.g., \cite[Section 2]{AW2}.
In this case we call it the {\it Coleff-Herrera product} $\mu^\gamma$ formed by the $\gamma_j$.  It is well-known, and was first proved by Dickenstein-Sessa and Passare,
that the annihilator
$\ann \mu^\gamma=\{\phi\in\Ok;\ \phi \mu^\gamma=0\}$ is
precisely equal to the ideal $(\gamma)$ generated by $\gamma_1, \ldots, \gamma_p$, see, \cite[Eq.\ (4.3)]{A9} for the setting used here.   
It follows by the dimension principle that $\mu^\gamma$ is in $\W^{Z_\gamma}$.  
 If $\omega$ is a holomorphic  $(N,0)$-form, therefore
$\mu^\gamma\w \omega$ is in $\CH^{Z_\gamma}$. 

Any  Coleff-Herrera current $\mu$ can be written locally as
$\mu=a\mu^\gamma\w\omega$ for such a tuple $\gamma$ and some holomorphic  function $a$,  see, e.g., \cite[Theorem 1.1]{A9}.
Thus the annihilator  
$\ann\mu$ is the kernel of the sheaf mapping $\Ok\to \Ok/(\gamma), \ \phi\mapsto a\phi$,
and hence $\ann\mu$ is coherent.

\smallskip
Let $S\to Y$ be a vector bundle. We say that $b\in ASM(Y,S)$ 
if there is a representation \eqref{oskar}, where $\omega$ is a smooth section of
$L\otimes\pi^* S$. 
The statements above have analogues for $S$-valued sections.  For instance,
if $S$ is a line bundle and $\gamma_j\in   ASM(Y,S)$,   then \eqref{chprodukt}  is 
an $S^{-p}$-valued current.

\section{Global Coleff-Herrera currents on $\Pk^N$}\label{svart2}

Let $\delta_x$ be interior multiplication by the vector field 
$$
\sum_1^N x_j\frac{\partial}{\partial x_j}
$$
on  $\C^{N+1}$
and recall that a differential form  $\xi$ on $\C^{N+1}\setminus\{0\}$ is projective, i.e., 
the pullback of a form
on $\P^N$, if and only if $\delta_x\xi=\delta_{\bar x}\xi=0$, where $\delta_{\bar x}$ is the 
conjugate of $\delta_x$. We will identify forms on $\P^N$ and projective forms. 
Notice that 
$$
\Omega=\delta_x(dx_0\w\ldots\w dx_N)
$$  
is a non-vanishing section of the trivial bundle over $\P^N$, realized as
a $(N,0)$-form on $\Pk^N$ with values in $\Ok(N+1)$.

\smallskip
Let $\gamma_1,\ldots, \gamma_p$ be holomorphic sections of $\Ok(r)$ such that their common zero set
$Z_\gamma$ has codimension $p$.  Then, cf., Section~\ref{prutman} above, 
\begin{equation}\label{chp}
\mu^\gamma\w\Omega=\dbar\frac{1}{\gamma_p}\w\cdots\w\dbar\frac{1}{\gamma_1}\w\Omega
\end{equation}
is a global section of $\CH^{Z_\gamma}\otimes \Ok(-pr+N+1)$.

\begin{lma}\label{skola} 
Let $Z\subset Z_\gamma$ be a reduced projective variety of pure codimension
$p$ and let $\mu$ be a global section of $\CH^Z\otimes\Ok(\ell+N+1)$ such that
 \begin{equation}\label{brum}
\gamma_1\mu=\cdots=\gamma_p\mu=0.
\end{equation}
If $p\le  N-1$, then there is a global  holomorphic section $a$ of
$\Ok(\ell+pr)$ such that 
\begin{equation}\label{snart}
\mu=a\dbar\frac{1}{\gamma_p}\w\cdots\w\dbar\frac{1}{\gamma_1}\w\Omega.
\end{equation}
If $p=N$ and $\ell+N\ge 0$,  then the same conclusion holds.
\end{lma}

In particular we see that if $p\le N-1$ and $\ell +pr<0$, then $\mu=0$. 

\begin{proof}
Let us introduce a trivial vector bundle $E$ of rank $p$ with global holomorphic frame elements
$e_1,\ldots,e_p$ and let $e_1^*,\ldots,e_p^*$ be the dual frame for $E^*$. We then have  the mapping
interior multiplication $\delta_\gamma\colon \Lambda^{*+1}E\to \Lambda^* E$  by the
section $\gamma:=\gamma_1 e_1^*+\cdots +\gamma_p e_p^*$ of $E^*$.  We  consider the exterior
algebra of $E\oplus T^*\P^N$ so that $d\bar x_j\w e_j=-e_j\w d\bar x_j$ etc.  
Then both $\delta_\gamma$ and $\dbar$ extend to mappings on currents with values in $\Lambda E$,
and 
\begin{equation}\label{sprud}
\delta_\gamma \dbar=-\dbar\delta_\gamma.
\end{equation}  
Let $e=e_1\w\ldots\w e_p$.
Recall that
$H^{N,k}(\P^N,\Ok(\nu))=0$ if either $1\le k\le N-1$ or $k=N$ and $\nu\ge 1$;
see, e.g., \cite[Ch.~VII, Theorem~10.7]{Dem}.
If $p\le N-1$, or $\ell+N+1\ge 1$,   we can  therefore find a global solution
to  $\dbar w_{p-1}=\mu\w e$.  In view of \eqref{sprud} and \eqref{brum} we have that
$$
\dbar \delta_\gamma w_{p-1}=-\delta_\gamma \dbar w_{p-1}=-\delta_\gamma (\mu\w e)=0.
$$
Thus we can   successively solve
\begin{equation}\label{postum}
\dbar w_{p-1}=\mu\w e,\  \dbar w_{p-2}=\delta_\gamma w_{p-1},\ldots,\dbar w_0=\delta_\gamma w_1.
\end{equation}
Then $a\w\Omega:=\delta_\gamma w_0$ is a $\dbar$-closed, and thus a holomorphic,  $(N,0)$-form
with values in $\Ok(\ell+pr+N+1)$.  Altogether, 
$$
(\delta_\gamma-\dbar)w=a\w\Omega-\mu\w  e
$$
if $w=w_0+\cdots+w_{p-1}$. 
As in \cite[Examples~ 3.1 or 3.2]{A9} we can find a global current $U$ such that 
$$
(\delta_\gamma-\dbar)U=1-\mu^\gamma\w e.
$$
Thus 
$$
(\delta_\gamma-\dbar)(a U\w\Omega-w)=\mu-a\mu^\gamma\w\Omega.
$$
Since  the right hand side is in $\CH^Z$ it now follows from \cite[Theorem~3.3]{A9} 
that it must vanish.  
\end{proof}

\begin{ex}\label{ros}  
Given a global section $\mu$ of $\CH_Z\otimes\Ok(\ell)$
one can always find $\gamma_j$ such that \eqref{brum} holds. 
In fact, for a large enough $r_0$ there are sections $g_1',\ldots,g_m'$ 
of $\Ok(r_0)$ that generate the ideal sheaf  $\J_Z\subset\Ok_{\Pk^N}$.  
If $g_1,\ldots,g_p$ are generic linear combinations of the $g_j'$, then
$Z_g=\{g_1=\cdots =g_p=0\}$ has codimension $p$,  $Z_g\supset Z$, and 
(expressed in a local frame) $dg_1\w\ldots\w dg_p\neq 0$ on $Z_{reg}$.   
If $\gamma_j=g_j^{\m_j+1}$ and $\m_j$ are large enough, then \eqref{brum} holds.
\end{ex}

\section{Bj\" ork-type representation of global Coleff-Herrera currents}\label{pyts}
In this section we express the action $\mu.\xi$ if a global Coleff-Herrera current $\mu$ on a test form $\xi$ as an integral over $Z$ of $\M\xi$,
where $\M$ is a certain differential operator.
 
\smallskip
As usual we identify smooth sections $\psi$ of the line bundle $\Ok(\ell)$ by $\ell$-homogeneous
smooth functions on $\C^{N+1}\setminus\{0\}$. 
Notice that then each $\partial/\partial x_j$, $j=0,\ldots,N$,  induces a differential operator
$\Ok(\ell)\to \Ok(\ell-1)$.   We say that a finite sum
\begin{equation}
\L=\sum_{\alpha} v_\alpha \frac{\partial^\alpha}{\partial x^\alpha}
\end{equation} 
is a holomorphic differential operator on $\Pk^N$ of {\it degree} $r$ if the coefficients $v_{\alpha}$
are holomorphic sections of $\Ok(r+|\alpha|)$.  Such an $\L$ maps
$\Ok(\ell)\to \Ok(\ell+r)$ for each $\ell$.   
The {\it order} of $\L$ is the maximal occurring $|\alpha|$ as usual.

Consider the affinization $\C^N\simeq \{x_0\neq 0\}$. Notice that there is a one-to-one
correspondence between smooth sections of $\Ok(\ell)$ over $\C^N$ and smooth functions in $\C^N$,
via the frame $[x_0,\ldots,x_N]\mapsto x_0^\ell$ for $\Ok(\ell)$ over $\C^N$.  More concretely,
given the section $\phi$ one gets the associated function by just letting $x_0=1$. Conversely,
given $\Phi$, then $\phi(x)=x_0^\ell\Phi(x'/x_0)$.  In this way a differential operator
of degree $r$ gives rise to a differential operator 
$$
L=\sum_{|\alpha'|\le M} V_{\alpha'}(x')\frac{\partial^{\alpha'}}{\partial x^{\alpha'}}
$$
where $V_{\alpha'}(x')$ are polynomials of degree at most $r+|\alpha'|$.
Notice however, that the resulting affine $L$ will depend on $\ell$ unless $\L (x_0\phi)=x_0\L\phi$
for all $\phi$. For instance, 
the differential operator $\L=\partial/\partial x_0$, that has order $1$ and degree $-1$,
induces
$$
L=\ell  -\sum_1^N x_j\frac{\partial^j}{\partial x^j}.
$$
Notice that $\L$, as well as an associated affine differential operator $L$, 
act on smooth $(0,*)$-forms as well.

The following statement is a global version of a construction due to Bj\"ork,
\cite{Bj}. A similar result is obtained in \cite[Theorem~4.2]{VY}.

\begin{thm}\label{skohorn}
Assume that $Z\subset\Pk^N$ has pure codimension $p$,  that $\mu$
is a global section of  $\CH_Z\otimes\Ok(r)$, and assume that
$p\le N-1$ or $r+1\ge 0$. Let $\I=\ann\mu$.
There is a multiindex $\m=(\m_1,\ldots,\m_p)$, a number $\rho$,
and for each $\alpha\le\m$ there are 
holomorphic differential operators $\L_\alpha$ and $\M_{\m-\alpha}$,
such that $\deg \L_\alpha+\deg\M_{\m-\alpha}=\rho$,
and a global meromorphic $(n,0)$-form $\tau$ with values in $\Ok(-\rho)$, 
not identically polar on any irreducible 
component of $Z$, such that the following hold:

\smallskip

\noindent(i) For any global holomorphic section $\phi$ of $\Ok(\ell)$ and
any test form $\xi$ of bidegree $(0,n)$ with values in $\Ok(-r-\ell)$ we have
\begin{equation}\label{skaldjur}
\phi\mu.\xi=
\sum_{\alpha\le m}\int_Z \tau\w   \L_\alpha\phi\w \M_{\m-\alpha}\xi.
\end{equation}

\smallskip
\noindent (ii) 
For each point $x\in Z$,
a germ $\psi\in\Ok_x$ is in $\I_x$ if and only if
\begin{equation}\label{boras}
\L_\alpha\psi\in \sqrt{\I_x},\quad \alpha\le \m.
\end{equation}

\smallskip
\noindent (iii)
For each $\alpha\le\m$ there are holomorphic differential operators $\M_{\alpha,\gamma}$, 
$\gamma\le\alpha$, such that
\begin{equation}\label{ormvrak2}
\L_{\alpha}(\phi\psi)=\sum_{\gamma\le\alpha}\L_\gamma\phi\M_{\alpha,\gamma}\psi
\end{equation}
for  all holomorphic sections $\phi$ and $\psi$ of $\Ok(\ell)$ and $\Ok(\ell')$.
\end{thm}

\begin{proof}
To begin with we choose $g_1,\ldots,g_p$, 
$\m:=(\m_1,\ldots,\m_p)$,  and $a$ as in Example~\ref{ros} and Lemma~\ref{skola} so that
\begin{equation}\label{plura}
\mu=a\mu^{g^{\m+\1}}\w\Omega.
\end{equation}
After a projective transformation on $\Pk^N$, i.e., a linear change of variables on $\C^{N+1}$, we may assume 
that each irreducible component of $Z$ intersects  the affine space $\C^N:=\{x_0\neq 0\}$. 
Then the affinizations $G_j$ of $g_j$  are
polynomials in $\C^N$ such that $dG_1\w\ldots\w dG_p$ is nonvanishing on
$Z_{reg}\cap\C^N$, cf., Example~\ref{ros}.
Let $x'=(x_1,\ldots,x_N)$. After possibly a linear transformation of $\C^N$, we may assume that
the polynomial
$$
H:= \det\frac{\partial G}{\partial \eta}
$$
is generically nonvanishing on $Z\cap\C^N$, where
$$
x'=(\zeta,\eta)=(\zeta_1,\ldots,\zeta_n,\eta_1,\ldots\eta_p).
$$ 
Let us introduce the short hand notation 
$$
\dbar\frac{1}{G^{\m+\1}}=\dbar\frac{1}{G_1^{\m_1+1}}\w\ldots\w\dbar\frac{1}{G_p^{\m_p+1}}.
$$
We first look for a representation of the Coleff-Herrera current
$$
\tilde\mu =\dbar\frac{1}{G^{\m+\1}}\w d\eta\w d\zeta
$$
at points $x$ on $Z':=Z\cap\C^N\cap\{H\neq 0\}$. Locally at such a point 
we can make the change of variables
$$
w=G(\zeta,\eta),\quad  z=\zeta.
$$
If $\Xi$ is a smooth $(0,n)$-form with small support, and $\Phi$ is holomorphic,
with the notation
$m!=m_1!\cdots m_p!$ and $\partial^\alpha_w=\partial^{|\alpha|}/\partial w^\alpha$, etc, 
in view of Example~\ref{brus} we then have
\begin{multline*}%
\Phi\tilde \mu.\Xi=\int\dbar\frac{1}{G^{\m+\1}}\w d\eta\w d\zeta\w\Phi\Xi=
\pm\int\dbar\frac{1}{w^{\m+\1}}\w dw\w dz\w \frac{\Xi}{H}\Phi=\\
\pm \int_{w=0}\frac{(2\pi i)^p}{\m!} dz\w \partial^\m_w\Big(\frac{\Xi}{H}\Phi\Big)=
\pm \sum_{\alpha\le \m}\int_{w=0}\frac{(2\pi i)^p}{(\m-\alpha)!\alpha!}
dz\w \partial^{\m-\alpha}_w\Big(\frac{\Xi}{H}\Big)
\partial^\alpha_w\Phi.
\end{multline*}
Now, notice that 
$$
\partial_\eta=(\partial_\eta G)\partial_w
$$
so that 
$$
\partial_w=\frac{\Gamma}{H}\partial_\eta,
$$
where $\Gamma$ is a matrix of polynomials. It is readily checked that
\begin{equation}\label{snar}
\tilde L_\alpha:=H^{2|\alpha|}\Big(\frac{\Gamma}{H}\partial_\eta\Big)^\alpha
\end{equation}
has a holomorphic extension across $H=0$. Let us define 
$$
M_\beta \Xi= \pm\frac{(2\pi i)^p}{\beta!(\m-\beta)!}H^{1+|\m|+2|\beta|}
\Big(\frac{\Gamma}{H}\partial_\eta\Big)^\beta\frac{\Xi}{H}.
$$
Then also $M_\beta $ is holomorphic across $H=0$.

With $T=dz=d\zeta$, we  have that
\begin{equation}\label{snor}
\Phi\tilde \mu.\Xi=\int_{Z'}\sum_{\alpha\le \m}\frac{T}{H^{3|\m|+1}}\w M_{\m-\alpha}\Xi
\w \tilde L_\alpha \Phi
\end{equation}
for $\Xi$ with support close to $x$.  
We claim that 
if $\Phi$ is a germ of a holomorphic function at $x$, then 
$\Phi\tilde\mu_x=0$ if and only if $\tilde L_\alpha\Phi=0$ on $Z_x$
for all $\alpha\le \m$. In fact, 
\begin{multline}\label{apa0}
\Phi\tilde\mu_x=0 \iff \Phi\dbar\frac{1}{G^{\m+1}}|_x=0 \iff
\Phi\dbar\frac{1}{w^{\m+1}}|_x=0 \iff  \\
\partial_w^\alpha\Phi=0 \ \text{on}\ Z_x, \ \alpha\le\m
\iff \tilde\L_\alpha\Phi =0 \ \text{on} \ Z_x, \ \alpha\le \m.
\end{multline}
\smallskip
Now, for each $\alpha\le m$, let us homogenize the coefficients in  $\tilde L_\alpha$ to obtain
$\tilde \L_\alpha$ for some fixed degree,  and then let us homogenize $M_{\m-\alpha}$ to $\M_{\m-\alpha}$ so that
the sum of their degrees is a fixed number $\rho$. Let $\tau'$ be the homogenization of $T=d\zeta$, i.e., 
$$
\tau'=d\frac{x_1}{x_0}\w\ldots\w d\frac{x_n}{x_0}
$$
if $x=(x_0,\ldots,x_N)=(x_0,\zeta,\eta)$. 
Finally let us homogenize $H^{3|\m|+1}$ to $h$ so that $\tau:=\tau'/h$ takes values in
$\Ok(-\rho)$. We possibly get some factors $x_0$ in the denominator, but since
$Z$ has no irreducible component in $\{x_0=0\}$ this is acceptable.

Let us define the global current 
\begin{equation}\label{plex}
\tilde\mu:=\1_Z\mu^{g^{\m+\1}}\w\Omega
\end{equation}
in  $\P^N$.  
In view of \eqref{plura} it  takes values in $\Ok(r-\deg a)$. At each point $x\in Z'$ it is the
$(r-\deg a)$-homogenization of our previous $\tilde\mu$ but the global current is not
necessarily $\dbar$-closed at $x_0$.   However, in view of \eqref{plura},  \eqref{jumper3}, and \eqref{plex}, 
\begin{equation}\label{plura2}
a\tilde\mu=a\1_Z\mu^{g^{\m+\1}}\w\Omega=\1_Z a\mu^{g^{\m+\1}}\w\Omega=\1_Z\mu=\mu,
\end{equation}
since $\mu$ has support on $Z$,  and thus $a\tilde\mu$ is $\dbar$-closed.

For holomorphic sections $\phi$ of $\Ok(\ell-\deg a)$ and
test forms $\xi$ of bidegree
$(0,n)$ with support in $\P^N\setminus\{h=0, \ x_0=0\}$ and values in $\Ok(-r-\ell)$ we have
\begin{equation}\label{glob1}
\phi\tilde\mu.\xi=\int_Z \sum_{\alpha\le \m}\tau \w \M_{\m-\alpha}\xi
\w \tilde\L_\alpha \phi.
\end{equation}
By Proposition~\ref{jumper2},  $\tau\w\tilde L \phi\w [Z]$
is a global section of
$\W^Z\otimes \Ok(e+\ell)$ 
and thus 
the integrals on the right hand side
of  \eqref{glob1} exist as a principal values for any test form $\xi$.
In view of Proposition~\ref{spadtag} the right hand side of \eqref{glob1}
defines the action on $\xi$ of a global section of 
$\W^Z\otimes\Ok(e+\ell)$.  
Since $\{h=0, \ x_0=0\}\cap Z$ has positive codimension on $Z$
it follows by the SEP that the equality \eqref{glob1} holds for all  $\xi$.
\smallskip

Define the holomorphic differential operators
$\L_\alpha$ by the equality
\begin{equation}\label{apsko}
\L_\alpha\phi=\tilde\L_\alpha(a\phi).
\end{equation}
Then \eqref{skaldjur} follows from \eqref{glob1}. Thus (i) is proved.

\smallskip
For $x\in Z'=Z\setminus\{h=0,\ x_0=0\}$ we have, by \eqref{apa0} and \eqref{apsko},  that
\begin{equation}\label{apa2}
\phi \mu_x=0\  \text{if\ and\ only\ if}\   \L_\alpha\phi=0 \ \text{on}\  Z_x, \ \alpha\le\m.
\end{equation}
Again since $\{h=0, \ x_0=0\}\cap Z$ has positive codimension on $Z$, it follows by
continuity and the SEP that \eqref{apa2} holds for all $x\in Z$. Thus (ii) is proved.

\smallskip
To see (iii),  just notice that 
$$
\tilde L_{\alpha}(\Phi\Psi)=\sum_{\gamma\le\alpha}L_\gamma\Phi c_{\alpha,\gamma} L_{\alpha-\gamma}\Phi,
$$
where $c_{\alpha,\gamma}$ are binomial coefficients. After homogenization and replacing
$\phi$ by $a\phi$ we get (iii) with $\L_{\alpha,\gamma}= c_{\alpha,\gamma} \L_{\alpha-\gamma}$.
\end{proof}

\begin{remark}
One can  check, cf., \cite[Section~5]{AW2},  that 
$\phi\1_Z\tilde\mu=0$ if and only if $\phi$ is in the
intersection of the primary ideals of $(g^{\m+1})$ associated with the 
irreducible components of $Z$.
\end{remark}

Let $\mu$ be a global section of $\CH^Z\otimes\Ok(r)$ in $\P^N$ and let $b$ be a global almost
semi-meromorphic current of bidegree $(0,*)$
with values in  $\Ok(r_1)$. Then $b\mu$ is a section of
$\W^Z\otimes\Ok(r+r_1)$.
Let us also assume that  $ZSS(b)\cap Z$ has positive codimension
in $Z$. 
Consider a representation of $\mu$
as in Theorem~\ref{skohorn}. In view of Theorem~\ref{jumper1} 
we can define
differential operators $\widehat M_\gamma$  with almost semi-meromorphic coefficients 
so that 
$$
\widehat M_\gamma\xi= M_\gamma( b \xi).
$$
For test forms $\xi$ of bidegree $(0,*)$ with values in $\Ok(-r-\ell)$ 
and with support outside $ZSS(b)$,
and any global holomorphic section $\phi$ of $\Ok(\ell)$ 
we  have
\begin{equation}\label{skaldjur2}
\phi b\mu.\xi=
\sum_{\alpha\le \m}\int_Z \tau\w   \L_\alpha\phi\w \widehat M_{\m-\alpha}\xi.
\end{equation}
In  view of Propositions~\ref{jumper2} and \ref{spadtag} the right hand side
defines a global section of $\W^Z\otimes\Ok(r+r_1)$.
Since $Z\cap ZSS(b)$ has positive codimension in $Z$, it 
follows that \eqref{skaldjur2} holds globally.

\section{Proof of Theorem~\ref{huvudsats}}\label{grotesk}

Let $X$ be our non-reduced subspace of $\Pk^N$. As was mentioned in the introduction the proof relies on the global 
current  $R^f\w R^X$  that we first discuss.

\subsection{The current $R^X$}
Given a vector bundle $E\to\P^N$, let $\Ok(E)$ denote the associated locally free
analytic sheaf.
We can find a locally free resolution
$$
0\to\Ok(E_{N})\stackrel{c_N}{\to} \cdots \stackrel{c_2}{\to} \Ok(E_1) \stackrel{c_1}{\to}\Ok(E_0)\to\Ok_{\Pk^N}/\J_X\to 0
$$
of $\Ok_{\Pk^N}/\J_X$, where $E_0$ is a trivial line bundle and
$
E_k=\oplus_i^{r_k}\Ok(-d_k^i)
$
for suitable positive numbers $d_k^i$,  
see, e.g.,  \cite{AWsemester}.  In fact, we can use the "same" mappings 
$c_k=(c_k^{ij})$
as in \eqref{plast} but with $c_k^{ij}$ considered as sections of   
$\Ok(d^j_k-d^{i}_{k-1})$. 
There is a natural choice of Hermitian metrics on $E_k$ and following 
\cite[Sections~3 and 6]{AW1} there is an associated current
$$
R^{X}=R_p^{X}+\cdots+R_{N}^{X}
$$
with support on $X_{red}$, where $R_k^{X}$ are $(0,k)$-currents that
take values in $E_k$, and with the property that 
$\phi R^{X}=0\  \text{if\ and\ only\ if} \  \phi\in\J_X.$   
Furthermore, 
\begin{equation}\label{pelly}
\dbar R^X_k= c_{k+1}R^X_{k+1}, \quad k\ge 0.
\end{equation}

\begin{prop}
There is a bundle 
\begin{equation}\label{fdef}
F=\oplus_{i=1}^{r_F}\Ok(d_F), 
\end{equation}
a global
section $\mu$ of $\CH^{X_{red}}\otimes F\otimes\Ok(N+1)$, and an almost semi-meromorphic section $b$ of
$\Hom(F,\oplus_{i=p}^{N+1}E_k)$ such that
\begin{equation}\label{plonk}
R^{X}\w\Omega= b\mu.
\end{equation}
in $\Pk^N$.
\end{prop}

\begin{proof} Since the kernel $\K$ of $c_{p+1}^*\colon \Ok(E^*_p){\to}\Ok(E^*_{p+1})$ is coherent, 
for a large enough integer $d_F$, $\K\otimes\Ok(d_F)$ is generated by global sections
$g_1,\ldots, g_{r_F}$. We therefore have a surjective sheaf mapping
$\oplus_1^{r_F}\Ok\to\K\otimes\Ok(d_F)$ and hence
$\oplus_1^{r_F}\Ok(-d_F)\to\K$. Define $F$ by \eqref{fdef} 
and let $g\colon \Ok(E_p)\to \Ok(F)$
be the dual of the composed  mapping $\Ok(F^*)\to\K\to \Ok(E^*_p)$.
We then have the exact sequence
\begin{equation}\label{anka}
\Ok(F^*)\stackrel{g^*}{\to}\Ok(E^*_p)\stackrel{c^*_{p+1}}{\to}\Ok(E^*_{p+1})
\end{equation}
of sheaves. 
We  claim that
$$
\mu:=g R^X_p\w\Omega
$$
is a global (vector-valued) Coleff-Herrera current.  
In fact, in view of \eqref{pelly}, 
$$
\dbar \mu=\dbar g R^X_p\w\Omega=g\dbar R^X_p\w\Omega=
g c_{p+1} R^X_{p+1}\w\Omega=0,
$$
since $g c_{p+1}=0$. Because of the dimension principle $\mu$ must have the SEP
with respect to $X_{red}$ and hence it is, by definition, a Coleff-Herrera current
and thus a section of $\CH^{X_{red}}\otimes F \times \Ok(N+1)$.

Let $X_{p+1}$ be the subset of $X_{red}$ where $s_{p+1}$ does not have
optimal rank. Let us choose a Hermitian norm on $F$, and define 
$\sigma_F\colon F\to E_p$ on the complement of $Z_{p+1}$ so that 
$\sigma_F=0$ on the orthogonal complement of $\Im g$ and 
$\sigma_F g =I$ on the orthogonal complement of $\Ker g$.   It is shown in
\cite[Section 2]{AW2} that $\sigma_F$ has an almost semi-meromorphic
extension across $X_{p+1}$; let us denote the extension by $\sigma_F$ as well. 
Following the proof of \cite[Proposition~3.2]{Sz} we see (this is just a local
argument) that $R^X_p=\sigma_F g R^X_p$ outside $X_{p+1}$. The right hand
side here is defined in view of Theorem \ref{jumper2}.   Since both sides have the
SEP on $X_{red}$ we conclude that they coincide in $\Pk^N$.  Thus
\begin{equation}\label{polly}
R^X_p\w\Omega=\sigma_F \mu.
\end{equation}
From \cite[Theorem 4.4]{AW1} we get global almost semi-meromorphic
sections $\alpha_{k+1}$ of $\Hom(E_k, E_{k+1})$, $k=p, p+1,\ldots$,
that are smooth outside analytic subsets $X_{k+1}$ of $X_{red}$ where
$s_{k+1}$ do not have optimal rank, such that
$$
R^X_{k+1}=\alpha_{k+1} R^X_k.
$$
Since $X$ has pure dimension it follows that
$\codim X_{p+\ell}\ge p+\ell+1$ according to \cite[Corollary 20.14]{Eis}. 
Arguing as in the proof of \cite[Proposition~3.2]{Sz} we now get for each
$k\ge p+1$, in view of 
\eqref{polly},  the representation 
\begin{equation}\label{pretty}
R^X_k=\alpha_k\cdots \alpha_{p+1}\sigma_F \mu.
\end{equation}
Now let $b_k=\alpha_k\cdots \alpha_{p+1}\sigma_F$. Then $b_k$ is
an almost semi-meromorphic, see \cite[Section 3.1]{AW3},
and by \eqref{pretty},
$R^X_k=b_k\mu$ where $b_k$ is smooth, that is, outside $Z_{p+1}$.
Since $\1_{Z_p+1}\mu=0$ it follows from \eqref{jumper3} that 
$R^X_k=b_k\mu$. Thus the proposition follows with $b=b_p+\cdots+b_N$. 
\end{proof}


\subsection{The current $R^a\w R^X$} 
Assume that we have sections $a_1,\ldots, a_m$ of a Hermitian
line bundle $S$ over some open set $\U\subset\Pk^N$ and
let $E$ be a trivial rank $m$ bundle. Then we have  
interior multiplication 
$\delta_a\colon \Lambda ^{*+1}E\otimes S^{-*-1}\to 
\Lambda ^{*}E\otimes S^{-*}$,
and we can consider the induced double complex as in the proof of 
Lemma \ref{skola} above. 
Following \cite[Example~2.1]{AWsemester} we define the Bochner-Martinelli form
$U^a=U^a_1+\cdots +U^a_N$,  explicitly from the $a_j$. 
The components $U^a_k$ are almost semi-meromorphic 
$(0,k-1)$-forms with values in $\Lambda ^{k}E\otimes S^{-k}$
that are smooth outside the common zero set $Z_a$ of the $a_j$.
Moreover, $(\delta_a-\dbar)U^a=1$ outside $Z_a$.  We thus have the residue current
$$
R^a:= 1-(\delta_a-\dbar)U^a,
$$
with support on $Z_a$, whose components $R^a_k$ are $(0,k)$-currents with values in
$\Lambda ^{k}E\otimes S^{-k}$.  
If $\chi_\epsilon=\chi(|a|^2/\epsilon)$,  where $\chi$ is a function as in \eqref{smorgas} above, then $U^{a,\epsilon}=\chi_\epsilon U^a$ are smooth and tend
to $U^a$.  Thus 
$$
R^{a,\epsilon}=1-(\delta_a-\dbar)U^{a,\epsilon}=
1-\chi_\epsilon +\dbar\chi_\epsilon\w U^a
$$
tend to $R^a$.  
As in \cite[Section~2.5]{AWsemester}, cf., \eqref{hoppsan} above,
we can form the product
\begin{equation}\label{allan}
R^a\w R^X\w\Omega:=\lim_{\epsilon\to 0} R^{a,\epsilon}\w R^X\w\Omega.
\end{equation}
We will use the following important property, which follows from 
\cite[(2.19)]{AWsemester} and the proof  
\cite[Lemma 2.2]{AWsemester}:

\begin{lma}\label{pixbo}
If $\Phi$ is holomorphic
 and 
$
\Phi R^a\w R^X\w\Omega=0
$
at $x$, then  $\Phi$ is in $(a)_x+\J_{X,x}$.
\end{lma}  

\begin{remark}[Warning!]  Although the components $R^a_k$ 
of $R^a$ vanish for small $k$ because of the dimension principle, 
the  terms  $R^a_k\w R^X$ might be nonzero.  See, e.g.,  \cite{AW3} for examples.
\end{remark}


\subsection{End of proof of Theorem \ref{huvudsats}}

To begin with we assume that $p=\codim Z\le N-1$.
Let $\mu$ be the (vector-valued) Coleff-Herrera current in the 
representation \eqref{plonk} of  $R^X\w\Omega$.
Let us consider $\mu$ as an $r_F$-tuple of Coleff-Herrera currents, and let
$\L_\alpha$, $\alpha\le\m$, be a (tuple of) Noetherian operators obtained from 
Theorem~\ref{skohorn}. Moreover,
let $\widehat M_\alpha$ be the associated differential operators with almost semi-meromorphic
coefficients so that \eqref{skaldjur2} holds.

\smallskip
At a given point $x\in X_{red}$ there is a number $\nu_x$ such that
if $(a)=(a_1,\ldots, a_m)\subset\Ok_{X,x}$ is a local ideal,
and $\phi\in\Ok_{X,x}$, then
$|\L_\alpha\phi|\le C|a|^{\nu}$ on $X_{red,x}$ for all $\alpha\le\m$ implies that 
$\phi R^a\w R^X\w\Omega=0$.  This is precisely 
the main step of the proof of \cite[Theorem~1.2]{Sz} and we do not repeat
it here (just notice that our number $\nu_x$ is called $N$ in \cite{Sz},
our $\widetilde M_\alpha$ are called $\tilde K_\alpha$, 
moreover,  the non-reduced space that we call $X$
is denoted by $Z$  in \cite{Sz}  whereas
$X$ denotes the associated reduced space!). In this proof the 
number $\nu_x$ is explicitly deduced from the singularities of the the coefficients of
$\widehat M_\alpha$ and of $b$, expressed as the degree of monomials
in a suitable log resolution of $X_{red}$,  see \cite[Eq.\ (4.9)]{Sz}.
In particular, the number $\nu_x$
works for all points in a \nbh of $x$.  By compactness we therefore get:

\begin{prop}\label{daggmask} 
There is a number $\nu$,  such that if $x\in X_{red}$, 
$(a)=(a_1,\ldots, a_m)\subset\Ok_{X,x}$ is a local ideal,
and $\phi\in\Ok_{X,x}$, then
$|\L_\alpha\phi|\le C|a|^{\nu}$ on $X_{red,x}$ for all $\alpha\le\m$ implies that 
$\phi R^a\w R^X\w\Omega=0$.
\end{prop}

Combined with  Lemma \ref{pixbo} we have thus obtained
 $\nu$ and differential operators $\L_\alpha$  so that
part (i) of Theorem~\ref{huvudsats} holds.

 
\smallskip
Now let $F_j$ be polynomials as in Theorem~\ref{huvudsats} (ii), let
$f_j$ be the $d$-homogenizations considered as section of $\Ok(d)$ over $X_{red}$ and 
let $\J_f$ be the associated ideal sheaf as in the introduction.  

\begin{lma}\label{sot} 
Let  $\Phi$ be a polynomial such  that \eqref{kaka} holds and let $\phi$
be the $\rho$-homogenization of $\Phi$. 
If 
\begin{equation}\label{palsdjur}
\rho\ge \deg\Phi+\nu d^{c_\infty}\deg X_{red},
\end{equation}
then 
$
|\L_\alpha\phi|\le C|f|^\nu
$
for all $\alpha$.
\end{lma}

\begin{proof}
Let $\pi\colon \tilde X\to X_{red}$ be the normalization of the blow-up of $X_{red}$ along $\J_f$ and let
$\sum r_jW_j$ be the exceptional divisor, where $W_j$ are the 
irreducible components and  $r_j$ the corresponding multiplicities. 
Notice that if $\psi$ is a holomorphic section of some $\Ok(\ell)$, then
$|\psi|\le C|f|^\nu$ if and only if $\pi^*\psi$ vanishes to order at 
least $\nu r_j$ on $W_j$ for each $j$. 

If \eqref{kaka} holds on $V_{red}$, then  $\pi^*(\L_\alpha\phi)$ vanishes
to order $\nu r_j$ on each $W_j$ that is not fully contained in $\pi^{-1}(X_{red,\infty})$.
Notice that
$$
\phi=x_0^{\rho-\deg\Phi}\varphi,
$$
where $\varphi$ is the $\deg\Phi$-homogenization of $\Phi$ and thus holomorphic. 
If $W_j$ is contained in
$\pi^{-1}X_{red,\infty}$, then $\phi$ vanishes at least to order $\rho-\deg\Phi$ on $W_j$. 
Since $\L_\alpha$ does not involve the derivative $\partial/\partial x_0$
also $\L_\alpha \phi$ vanishes to order  $\rho-\deg\Phi$ on $W_j$.
By the geometric estimate in \cite{EL}, cf., \cite[Eq.~(6.2)]{AWsemester},  we have that
$$
r_j\le d^{\codim \pi(W_j)}\deg X_{red}.
$$
If \eqref{palsdjur} holds, 
therefore $\pi^*(\L_\alpha\phi)$ vanishes, at least,  to order
$\nu r_j$ on $W_j$ for all $j$. Thus the lemma follows.
\end{proof}

With the same hypotheses as in Lemma~\ref{sot}
it follows from the lemma and Proposition~\ref{daggmask}
that 
\begin{equation}\label{sot2}
\phi R^f\w R^X\w\Omega=0.
\end{equation}
If in addition 
$$
\rho\ge (d-1)\min(m,n+1)+\reg X,
$$ 
we can now solve a sequence of global
$\dbar$-equations in $\Pk^N$ and get a global solution  $q_j$ to
$\phi=f_1q_1+\cdots+f_mq_m$, cf.,  \cite[Lemma~4.3]{AWsemester}.  
The fact that $X$ is not reduced plays no role here.     
After dehomogenization we obtain the desired representation of $\Phi$, and so the
proof of Theorem~\ref{huvudsats} is complete in case $p\le N-1$.

\smallskip
Now assume  that $p=\codim Z=N$ so that $X_{red}$ is a finite set in $\C^N\simeq \P^N\setminus\{x_0=0\}$.
If necessary we  multiply $\mu$  by a suitable power
of $x_0$ to be able to apply Theorem~\ref{skohorn}. We then get the global, in $\C^N$,
$L_\alpha$ that form a complete set 
of Noetherian operators at each point $x\in X_{red}$.  Part (ii) is trivial,
since the image of any ideal $(a)\subset\Ok_{X,x}$ in $\Ok_{X_{red},x}$ is just
either $(0)$ or $(1)=\Ok_{X_{red},x}$.




\def\listing#1#2#3{{\sc #1}:\ {\it #2},\ #3.}

\end{document}